\documentclass{article}

\usepackage{hyperref}

\usepackage[accepted]{icml2020}
\usepackage[utf8]{inputenc} 
\usepackage[T1]{fontenc}    
\usepackage{hyperref}       
\usepackage{url}            
\usepackage{booktabs}       
\usepackage{amsfonts}       
\usepackage{nicefrac}       
\usepackage{microtype}      

\usepackage{amsthm}
\usepackage{amsmath}
\usepackage{amssymb}
\usepackage{makecell}
\usepackage{tikz}
\usetikzlibrary{fit,automata,positioning,calc}

\usetikzlibrary{positioning}
\usetikzlibrary{shapes,fit}
\tikzset{main node/.style={circle,draw,minimum size=1cm,inner sep=2pt}}
\usetikzlibrary{arrows}

\usepackage{pifont}

\usepackage{graphicx}

\hypersetup{
   colorlinks=true,%
   citecolor=black,%
   filecolor=black,%
   linkcolor=black,%
   urlcolor=blue
}

\usepackage{algorithm}
\usepackage{algorithmic}
\usepackage{enumitem}

\usepackage{constants}

\newconstantfamily{g}{
symbol=g,
}
\newconstantfamily{R}{
symbol=R,
}

\def \l{\left}
\def \r{\right}
\def \R{\mathbb{R}}

\def \E{\mathbb{E}}

\def \ol{\overline}
\def \b{\beta}

\def \g{\gamma}
\def \N{{\operatorname{N}}}
\def \Row{\text{R}}
\def \d{\delta}
\def \t{\theta}
\def \tr{\operatorname{trace}}
\def \vol{\text{d}}
\def \wt{\widetilde}
\def \s{\sigma}
\def \KL{\operatorname{KL}}
\def \W{\operatorname{W}}
\def \TV{\operatorname{TV}}
\def \der{\operatorname{d}\hspace{-1pt}} 
\def \o{\rho}
\def \normal{\operatorname{normal}}
\def \A{\mathcal{A}}
\def \P{\mathcal{P}}
\def \diam{\operatorname{diam}}
\def \F{\mathcal{F}}
\def \relint{\operatorname{relint}}
\def \aff{\operatorname{aff}}
\def \nc{\operatorname{NC}} 
\def \lam{\lambda}

\DeclareMathOperator*{\argmin}{arg\,min}

\makeatletter
\newtheorem{thm}{\protect\theoremname}
  \newtheorem{lem}[thm]{\protect\lemmaname}
  
  \newtheorem{defn}[thm]{\protect\definitionname}

    \newtheorem{assumption}[thm]{\protect\assumptionname}

\makeatother

  \providecommand{\definitionname}{Definition}
  \providecommand{\examplename}{Example}
  \providecommand{\lemmaname}{Lemma}
  \providecommand{\corrolaryname}{Corollary}
  \providecommand{\propositionname}{Proposition}
  \providecommand{\conditionsname}{Conditions}
\providecommand{\theoremname}{Theorem}
\providecommand{\assumptionname}{Assumption}

\begin{document}

%

%

\twocolumn[

\icmltitle{Double-Loop Unadjusted Langevin Algorithm}

\begin{icmlauthorlist}
\icmlauthor{Paul Rolland}{epfl}
\icmlauthor{Armin Eftekhari}{umea}
\icmlauthor{Ali Kavis}{epfl}
\icmlauthor{Volkan Cevher}{epfl}
\end{icmlauthorlist}

\icmlaffiliation{epfl}{LIONS, Ecole Polytechnique F\'ed\'erale de Lausanne, Switzerland}
\icmlaffiliation{umea}{Department of Mathematics and Mathematical Statistics, Umea University, Sweden}

\icmlcorrespondingauthor{Paul Rolland}{paul.rolland@epfl.ch}

\vskip 0.3in
]

\printAffiliationsAndNotice{}

\begin{abstract}


A well-known first-order method for sampling from  log-concave probability distributions is the Unadjusted Langevin Algorithm (ULA). This work 
proposes a new annealing step-size schedule for ULA, which allows to prove new convergence guarantees for sampling from a smooth log-concave distribution, which are not covered by existing state-of-the-art convergence guarantees. To establish this result, we derive a new theoretical bound that relates the Wasserstein distance to total variation distance between any two log-concave distributions that complements the reach of Talagrand $T_2$ inequality. Moreover, applying this new step size schedule to an existing constrained sampling algorithm, we show state-of-the-art convergence rates for sampling from a constrained log-concave distribution, as well as improved dimension dependence.  


\end{abstract}


\section{Introduction}
Let $\der\mu^*(x) \propto e^{-f(x)}\der x$ be a probability measure over $\R^d$, where $f:\R^d \rightarrow \R$ is a convex function with Lipschitz continuous gradient. In order to sample from such distributions,
first-order sampling schemes based on the discretization of Langevin dynamics and, in particular the Unadjusted Langevin Algorithm (ULA), have found widespread success in various applications~\cite{welling2011bayesian, li2016scalable, patterson2013stochastic, li2016preconditioned}. 

An ever-growing body of literature has been devoted solely to the study of  ULA and its variations~\cite{ahn2012bayesian, chen2015convergence, cheng2017convergence, cheng2017underdamped, dalalyan2017user, durmus2017nonasymptotic, durmus2018analysis, dwivedi2018log, luu2017sampling, welling2011bayesian,ma2015complete}.
%
The ULA iterates are given as 
\begin{equation}
x_{k+1} = x_k - \gamma_{k+1} \nabla f(x_k) + \sqrt{2\gamma_{k+1}} g_k,
\label{Langevin_iter_intro}
\end{equation}
where $\nabla f$ is the gradient of $f$, $\{\gamma_k\}_{k\geq 0}$ is a non-increasing sequence of positive step-sizes, and the entries of $g_k\in \R^d$ are zero-mean and unit-variance Gaussian random variables, independent from each another and everything else.
In its standard form~\eqref{Langevin_iter_intro}, ULA can provably sample from any log-concave and smooth probability measure ~\cite{durmus2017nonasymptotic, durmus2018analysis}.

The recent analysis of ~\cite{durmus2018analysis} studies ULA through the lens of convex optimization. Their analysis shows strong resemblance with the convergence analysis of stochastic gradient descent (SGD) algorithm for minimizing a convex continuously differentiable function $f:\R^d \rightarrow \R$. Starting from $x_0 \in \R^d$, SGD iterates similarly as~\eqref{Langevin_iter_intro}:
\[
x_{k+1} = x_k - \gamma_{k+1} \nabla f(x_k) + \gamma_{k+1} \Theta(x_k),
\]
where $\{\gamma_k\}_{k\geq 0}$ is a non-increasing sequence of positive step-sizes, and $\Theta : \R^d \rightarrow \R^d$ is a stochastic perturbation to $\nabla f$. One way of proving convergence guarantees for this method is to show the following inequality:
\begin{equation}
\begin{split}
2 \gamma_{k+1}(\E[f(x_{k+1})] - f(x^*))& \leq \E\left[ \|x_k - x^*\|_2^2\right]  \\
&- \E \left[\|x_{k+1} - x^*\|_2^2\right] + C \gamma_{k+1}^2
\end{split}
\label{eq:descent-lemma-SGD}
\end{equation}
for some constant $C \geq 0$, $\forall k \geq 0$ and $x^* \in \argmin_{x \in \R^d} f(x)$. From this inequality, and using step size $\gamma_k \propto \frac{1}{\sqrt{k}}$, it is then possible to show convergence, in expectation, of the average iterate $\bar{x}_T = \frac{1}{T} \sum_{t=0}^{T-1} x_t$ to the optimal value, i.e., $\E[f(\bar{x}_T)] - f(x^*) = \mathcal{O}\left(\frac{1}{\sqrt{T}}\right)$.

In their paper, ~\cite{durmus2018analysis} showed a similar descent lemma as ~\eqref{eq:descent-lemma-SGD} for the sequence of generated measures $\{\mu_k\}_{k\geq 0}$ denoting the distributions of the iterates $\{x_k\}_{k\geq 0}$ in ~\eqref{Langevin_iter_intro}, in which the objective gap $\E[f(x_k)] - f(x^*)$ is replaced with the Kullback-Leibler divergence $\KL(\mu_k; \mu^*)$, and the Euclidean distance $\|x_k - x^*\|_2$ is replaced with the $2$-Wasserstein distance $\W_2(\mu_k, \mu^*)$:
\begin{equation}
\begin{split}
2 \gamma_{k+1} \KL(\mu_k; \mu^*) \leq & \W_2^2(\mu_k, \mu^*) - \W_2^2(\mu_{k+1}, \mu^*) \\
& + 2Ld \gamma_{k+1}^2,
\end{split}
\label{eq:descent-lemma-ULA}
\end{equation}
where $L$ is the Lipschitz constant of the gradient of $f$. Then again, using $\gamma_k \propto \frac{1}{\sqrt{k}}$, it is possible to show convergence of the average sample distribution $\bar{\mu}_T = \frac{1}{T} \sum_{t=0}^T \mu_t$ to $\mu^*$ in $\KL$ divergence, with rate $\mathcal{O}\left(\frac{d^3}{\sqrt{T}}\right)$.

In this work, we improve this convergence rate to $\mathcal{O}\left(\frac{d^3}{T^{\frac{2}{3}}}\right)$. To this end, we first establish a new bound that relates the $\W_2$ distance and the $\KL$ divergence between any two log-concave distributions. When applied to inequality~\eqref{eq:descent-lemma-ULA}, this new bound can be exploited to design a new step-size sequence $\{\gamma_k\}_{k \geq 0}$ that allows to derive new convergence rates for ULA.

We introduce a new multistage decaying step size schedule, which proceeds in a double loop fashion by geometrically decreasing the step-size after a certain number of iterations, and that we call Double-loop ULA (DL-ULA). To the best of our knowledge, all existing convergence proof for ULA use either constant, or polynomially decaying step sizes, i.e. of the form $\gamma_k = k^{-\alpha}$ for some $\alpha \geq 0$, and this is the first work introducing a multistage decaying step size for a sampling algorithm. Interestingly, there is precedence to support our approach in that such step decay schedule can improve convergence of optimization algorithms \cite{hazan2014beyond, ge2019step, aybat2019universally, yousefian2012stochastic}. 

Our new inequality relating $\KL$ divergence and $\W_2$ distance serves as an alternative to the powerful $T_2$ inequality \cite{gozlan2010transport}, the latter requiring stronger assumptions on the distributions. The literature on Langevin dynamics commonly proves the convergence of an algorithm in $\KL$ divergence and then extends it to the total variation ($\TV$) distance using the famous Pinsker's inequality ~\cite{pinsker1960information, cheng2017convergence, durmus2018analysis}. Our new inequality enables to do the same for extending convergence results to $\W_2$ distance in the case of general log-concave distributions, and hence, might be of independent interest. Note, however, that this inequality applied alone to extend the result of ~\cite{durmus2018analysis} to $W_2$ distance provides a suboptimal convergence rate, and modifying the step-size schedule and the analysis appears to be crucial for improving the rate.

Finally, we apply this multistage strategy to the constrained sampling algorithm MYULA \cite{brosse2017sampling}, which allows us to obtain improved convergence guarantees, both in terms of rate and dimension dependence. This approach provides state-of-the-art convergence guarantees for sampling from a log-concave distribution over a general convex set.

We summarize our contributions as follows:
\begin{itemize}[leftmargin=*]
\item We introduce a variant of the Unadjusted Langevin Algorithm, using a new multistage decaying step-size schedule as well as a clipping step. Our new approach, called DL-ULA, yields new convergence guarantees, that are not covered by existing convergence result (i.e., either better convergence rate or better dimension dependence compared to state-of-the-art results).
\item We apply our new step-size schedule to an existing Langevin-based constrained sampling algorithm, called MYULA \cite{brosse2017sampling}, and improve its convergence both in terms of iteration and dimension dependences. 
\item We introduce a new bound relating the $2$-Wasserstein and the $\TV$ distance between any two log-concave distributions.

\end{itemize}

A summary of our convergence rates can be found in Tables~\ref{table:unconstrained} and ~\ref{table:constrained}.

\paragraph{Road map}
In section $3$, we define several metrics on probability measures that we will use, recall some properties of log-concave distributions that we will exploit, as well as some results on convergence of ULA. In section $4$, we present our new extension of ULA for unconstrained sampling, by introducing a new multistage step size schedule. We then prove convergence guarantees by making use of a new bound relating the $\KL$ divergence and the $\W_2$ distance. Finally, in section $5$, we apply this procedure to the existing algorithm MYULA for constrained sampling, and show that it yields improved convergence guarantees, both in terms of convergence rate and dimension dependence.

\section{Related work}


\paragraph{Unconstrained sampling}
Sampling algorithms based on Langevin dynamics have been widely studied \cite{ahn2012bayesian, chen2015convergence, cheng2017convergence, cheng2017underdamped, dalalyan2017user, durmus2018analysis, dwivedi2018log, durmus2017nonasymptotic, luu2017sampling, welling2011bayesian}. Although most convergence rates have been established in the strongly log-concave setting, rates have also been shown for general log-concave distributions, and in particular exhibit larger dimension dependences (see Table~\ref{table:unconstrained}).

Convergence guarantees for ULA applied to a general unconstrained log-concave distribution have been successively improved over the years. To the best of our knowledge, the best existing convergence results are the one obtained by ~\cite{durmus2018analysis} and ~\cite{durmus2017nonasymptotic}, that respectively show $\mathcal{O}(d^{3}\epsilon^{-4})$ and $\mathcal{O}(d^{5}\epsilon^{-2})$ convergence guarantees in $\TV$ distance. In this paper, we improve upon the former one, by showing a $\mathcal{O}(d^{3}\epsilon^{-3})$ convergence rate. This result is not absolutely better than the one of ~\cite{durmus2017nonasymptotic}, but enjoys better dimension dependence.

Until recently, convergence rate in Wasserstein distance had not been proven in the general log-concave setting. Only recently, ~\cite{zou2018stochastic} presented a method based on underdamped Langevin dynamics that provably converges in $\W_2$-distance for a general log-concave distribution. 

In ~\cite{zou2018stochastic}, the authors show a $\mathcal{O}(d^{5.5}\epsilon^{-6})$ convergence rate in $\W_2$ distance for general log-concave distributions. However, they make the assumption that $\E_{X\sim \mu}\l[\|X\|_2^4\r] \leq \bar{U}d^2$ for some scalar $\bar{U}$. However, let $\der \mu(x) \propto e^{-\|x\|_2} \der x$, which is a log-concave distribution. Then, $\E_{X\sim \mu}\l[\|X\|_2^4\r] = \Omega(d^4)$ and their assumption does not hold. For comparison purpose, if we replace it with our weaker Assumption~\ref{assumption:lighttail}, their rate becomes $\mathcal{O}(d^{10.5}\epsilon^{-6})$.



\paragraph{Constrained sampling}
Extensions of ULA have been designed in order to sample from constrained distributions \cite{bubeck2018sampling, brosse2017sampling, hsieh2018mirrored, patterson2013stochastic}. In \cite{bubeck2018sampling}, the authors propose to apply ULA, and project the sample onto the constraint at each iteration. They show a convergence rate of $\mathcal{O}(d^{12}\epsilon^{-12})$ in $\TV$ distance for log-concave distributions (i.e., $\mathcal{O}(d^{12}\epsilon^{-12})$ iterations of the algorithm are sufficient in order to obtain an error smaller than $\epsilon$ in $\TV$ distance).

In \cite{brosse2017sampling}, the authors propose to smooth the constraint using its Moreau-Yoshida envelope, and obtain a convergence rate of $\mathcal{O}(d^5\epsilon^{-6})$ in $\TV$ distance when the objective distribution is log-concave. To do so, they penalize the domain outside the constrain directly inside the target distribution via its Moreau-Yoshida envelop. 

The analysis of MYULA in \cite{brosse2017sampling} only holds when the penalty parameter is fixed and chosen in advance, leading to a natural saturation after a certain number of iterations. In this work, we extend this procedure using our Double-loop approach. This allows to obtain improved convergence both in terms of rate and dimension dependence, i.e., $\mathcal{O}(d^{3.5}\epsilon^{-5})$ in $\TV$ distance, and to ensure asymptotic convergence of the algorithm since the penalty is allowed to vary along the iterations.

The special case of sampling from simplices was solved in \cite{hsieh2018mirrored}, introducing Mirrored Langevin Dynamics (MLD). Their work relies on finding a mirror map for the given constraint domain, and then performing ULA in the dual space. However, this method requires strong log-concavity of the distribution in the dual space. Moreover, finding a suitable mirror map for a general convex set is not an easy task.


\section{Preliminaries}


\subsection{Various measures between distributions}

Let us recall the distances/divergences between probability measures which are used frequently throughout the paper. 
The Kullback–Leibler (KL) divergence between two probability measures $\mu,\nu$ on $\R^d$ is defined as 
\begin{align}
\KL(\mu;\nu) = \E_{\mu} \log ( \der \mu / \der \nu ),
\end{align}
assuming that $\mu$ is dominated by $\nu$. Their Total Variation (TV) distance is defined as 
\begin{align}
\| \mu - \nu\|_{\TV} = \sup_S | \mu(S) - \nu(S)|,
\label{eq:defn-TV}
\end{align}
where the supremum is over all measurable sets $S$ of $\R^d$.
Finally, the $2$-Wasserstein (or $\W_2$ for short) distance between $\mu$ and $\nu$ is defined as
\begin{align}
\W_2^2(\mu, \nu) = \inf_{\phi \in \Phi(\mu, \nu)} \int_{\R^d \times \R^d} \|x - y\|^2 d\gamma(x,y),
\label{eq:W-defined}
\end{align}
where $\Phi(\mu, \nu)$ denotes the set of all joint probability measures $\phi$ on $\R^{2d}$ that marginalize to $\mu$ and $\nu$, namely, for all measurable sets $A,B \subseteq \R^d$, $\phi( A \times \R^d ) = \mu(A)$ and $\phi(\R^d \times B) = \nu(B)$.

The main difference between $\W_2$ and $\TV$ distances is that $\W_2$ associates a higher cost when the difference between the distributions occurs at points that are further appart (in terms of Euclidean distance). Due to this property, errors occurring at the tail of the distributions (i.e., when $\|x\|_2 \rightarrow \infty$) can have a small impact in terms of $\TV$ distance, but a major impact in terms of $\W_2$ distance.

\subsection{Log-concave distributions and tail properties}

We start by recalling the basic property that we will assume on the probability measure. We will then present some known results about this class of measures which will be exploited in the convergence analysis of our algorithm.

\begin{defn}
We say that a function $f:\R^d \rightarrow \R$ has $L$-Lipschitz continuous gradient for $L \geq 0$ if $\forall x, y \in \R^d$,
\[
\|\nabla f(x) - \nabla f(y)\|_2 \leq L \|x-y\|_2.
\]
\end{defn}

\begin{defn}
We say a function $f:\R^d \rightarrow \R$ in convex if $\forall 0 \leq t \leq 1$ and $\forall x, y \in \R^d$, 
\[
f(tx + (1-t)y) \leq tf(x) + (1-t)f(y)
\]
\end{defn}

\begin{defn}
We say that probability measure $\mu \propto e^{-f(x)} \der x$ is logconcave if $f$ is convex. Moreover, we say that $\mu$ is $L$-smooth if $f$ has a $L$-Lipschitz continuous gradient.
\end{defn}

As mentioned previously, bounding the Wasserstein distance between two probability measures requires controlling the error at the tail of the distributions. In order to deal with such a distance without injecting large dependence in the dimension, we make the following assumption on the tail of the target distribution, which is quite standard when working with unconstrained non-strongly log-concave distributions \cite{durmus2018analysis, durmus2017nonasymptotic}:

\begin{assumption} \label{assumption:lighttail}
There exists $\eta > 0, M_\eta > 0$ such that for all $x \in \R^d$ such that $\|x\|_2 \geq M_\eta$,
\[
f(x) - f(x^*) \geq \eta \|x - x^*\|_2
\]
where $x^* = \argmin_{x\in \R^d} f(x)$. Without loss of generality, we will also assume $x^* = 0$ and $f(x^*) = 0$.
\end{assumption}

Note that in the case of a distribution constrained to a set $\Omega \subset \R^d$, this assumption is naturally satisfied with $\eta$ arbitrary, and $M_\eta = \text{diam}(\Omega)$ where $\text{diam}(\Omega)$ is the diameter of $\Omega$.

In order to see how this assumption transfers into a constraint on the tail of the distribution, we recall two following results shown in~\cite{durmus2018analysis} and ~\cite{lovasz2007geometry} respectively.

\begin{lem}
Let $X \in \R^d$ be a random vector from a log-concave distribution $\mu$ satisfying assumption~\ref{assumption:lighttail}. Then
\[
\E_{X \sim \mu}\l[\|X\|_2^2 \r] \leq \frac{2d(d+1)}{\eta^2} + M_\eta^2
\]
\label{lem:moment_bound}
\end{lem}

\begin{lem}
Let $X \in \R^d$ be a random vector from a log-concave distribution $\mu$ such that $\E\l[\|X\|_2^2\r] \leq C^2$. Then, for any $R > 1$, we have
\[
Pr \l(\|X\|_2 > RC \r) < e^{-R+1}
\]
\label{lem:moment-tail-property}
\end{lem}

It is thus possible to combine both lemmas to show that any distribution satisfying assumption~\ref{assumption:lighttail} necessarily has a sub-exponential tail. This property will allow us to control the Wasserstein distance in terms of the total variation distance.
\begin{lem}
Let $X$ be a random vector from a log-concave distribution $\mu$ satisfying assumption~\ref{assumption:lighttail}. Then, $\forall R > 1$,
\[
Pr\l(\|X\|_2 > R\sqrt{\frac{2d(d+1)}{\eta^2} + M_\eta}\r) < e^{-R + 1}
\]
\label{lem:light_tail}
\end{lem}

\subsection{Unadjusted Langevin Algorithm}

Finally, we recall the standard Unadjusted Langevin Algorithm as well as a very useful inequality bounding the $\KL$ divergence between the target distribution and the $k$-th sample distribution.

Consider the probability space $(\R^d,\mathcal{B},\mu^*)$, where $\mathcal{B}$ is the Borel sigma algebra and $\mu^*$ is the target distribution. Suppose that $\mu^*$ is log-concave and  dominated by the Lebesgue measure on $\R^d$, namely, 
\begin{align}
\label{mu-density}
\der \mu^*(x) = C e^{-f(x)} \der x,
\qquad \forall x\in S,
\end{align}
where $C$ is an unknown normalizing constant and the function $f:\R^d\rightarrow\R$ is convex and $\nabla f$ is $L$-Lipschitz continuous. We wish to sample from $\mu^*$ without calculating the normalizing constant $C$. 

A well-known scheme for sampling for such a distribution is called ULA. Initialized at $x_0 \in \R^d$, the iterates of ULA  are 
\begin{equation}
x_{k+1} = x_k - \gamma \nabla f(x_k) + \sqrt{2\gamma} g_k
\label{Langevin_iter}
\end{equation}
for all $k \geq 0$, where $\gamma >0$ is the step-size and the entries of $g_k\in \R^d$ are zero-mean and unit-variance Gaussian random variables, independent from each another and everything else. Let $\mu_k$ be the probability measure associated to iterate $x_k$, $\forall k \geq 0$.  It is well-known that ULA converges to the target measure in $\KL$ divergence. 

More specifically, for $n \geq n_\epsilon=\mathcal{O}(d^3L\epsilon^{-2})$ iterations,  we reach $\KL(\ol{\mu}_{n} ; \mu^*) \leq \epsilon$, where $\ol{\mu}_n = \frac{1}{n} \sum_{k=1}^n \mu_k$ is  the average of the probability measures associated to the iterates $\{x_k\}_{k=0}^n$~\cite{durmus2018analysis}. The averaging sum $\frac{1}{n} \sum_{k=1}^n \mu_k$ is to be understood in the sense of measures, i.e., sampling from the $\bar{\mu}_n$ is equivalent to choosing an index $k$ uniformly at random among $\{1,...,n\}$, and then sampling from $\mu_k$.

To prove this result, the authors showed the following useful inequality that we will exploit in our analysis: 
\begin{lem}
Suppose that we apply the ULA iterations~\eqref{Langevin_iter} for sampling from a smooth log-concave probability measure $\mu^* \propto e^{-f(x)} \der x$ with constant step-size $\gamma > 0$, starting from $x_0 \sim \mu_0$. Then, $\forall n > 0$,
\begin{equation}
\KL(\bar{\mu}_n; \mu^*) \leq \frac{\W_2^2(\mu_0, \mu^*)}{2 \gamma n} + Ld\gamma.
\label{descent_lemma}
\end{equation}
\end{lem}

\section{DL-ULA for unconstrained sampling}

In this section, we present a modified version of the standard ULA for sampling from an unconstrained distribution and provide convergence guarantees. This modified version of ULA involves a new step size schedule as well as a projection step. We will show that it allows to obtain improved convergence rate, as well as the first convergence rate in $\W_2$-distance for overdamped Langevin dynamics.

\subsection{DL-ULA algorithm}

We consider the problem of sampling  from a smooth and unconstrained probability measure $\mu^* \propto e^{-f(x)} \der x$, where $f:\R^d \rightarrow \R$ is differentiable. To this end, we apply the standard ULA in a double-loop fashion, and decrease the step size only between each inner loop. Moreover, each inner loop is followed by a projection step onto some Euclidean ball. The procedure is summarized in Algorithm~\ref{alg:DL_ULA}.

The projection step appears to be crucial in our analysis in order to control the tail of the sample distribution, which is necessary for bounding its Wasserstein distance to the target distribution.

\begin{algorithm}[!t]
   \caption{Double-loop Unadjusted Langevin Algorithm (DL-ULA)}
   \label{alg:DL_ULA}
\begin{algorithmic}
\STATE \textbf{Input}: Smooth unconstrained probability measure $\mu^*$, step sizes $\{\gamma_k\}_{k\ge 1}$, number of (inner) iterations $\{n_k\}_{k\ge 1}$, initial probability measure ${\mu}_0$ on~$\R^d$, and thresholds $\{\tau_k\}_{k \ge 1}$.
\STATE \textbf{Initialization:} Draw a sample $x_0$ from the probability measure ${\mu}_{0}$.
\FOR{$k=1, \ldots$}
\STATE $x_{k,0} \leftarrow x_{k-1}$
\FOR{$n=1,\ldots,n_k$}
\STATE $x_{k,n+1} \leftarrow x_{k,n} - \gamma_k \nabla f(x_{k,n}) + \sqrt{2\gamma_k} g_{k,n}$, where $g_{k,n} \sim \mathcal{N}(0,I_d)$.
\ENDFOR
\STATE $x_{k} \leftarrow x_{k,i}$, where $i$ is drawn from the uniform distribution on $\{1,\cdots,n_k\}$.
\IF{$\|x_{k}\|_2 > \tau_k$}
\STATE $x_{k}\leftarrow \tau_k x_{k}/\| x_{k}\|_2$. 
\ENDIF
\ENDFOR
\end{algorithmic}
\end{algorithm}

In the following sections, we derive the convergence rate for Algorithms~\ref{alg:DL_ULA}. The global idea for showing the convergence of this algorithm is to use the inequality~\eqref{descent_lemma} recursively between each successive outer loop. We denote as $\bar{\mu}_k$ the average distribution associated to the iterates of outer iteration $k$ just \emph{before} the projection step. Similarly, we denote as $\tilde{\mu}_k$ the same distribution \emph{after} the projection step.

Each outer iteration $k$ uses as a starting point a sample from the previous outer iteration $x_{k,0} \sim \tilde{\mu}_{k-1}$. Therefore, we can apply the inequality~\eqref{descent_lemma} to the outer iteration $k$ to obtain
\begin{equation}
\KL(\bar{\mu}_k; \mu^*) \leq \frac{\W_2^2(\tilde{\mu}_{k-1}, \mu^*)}{2 \gamma_k n_k} + Ld\gamma_k.
\label{DL_descent_lemma}
\end{equation}

In order to unfold the recursion, we must have a bound on $\W_2^2(\tilde{\mu}_{k-1}, \mu^*)$ in terms of $\KL(\bar{\mu}_{k-1}, \mu^*)$. Using the light tail property of log-concave distributions, it is easy to obtain a bound between  $\W_2^2(\tilde{\mu}_{k-1}, \mu^*)$ and  $\W_2^2(\bar{\mu}_{k-1}, \mu^*)$. However, it is not clear how to bound  $\W_2^2(\bar{\mu}_{k-1}, \mu^*)$ by $\KL(\bar{\mu}_{k-1}, \mu^*)$.

As an intermediate step in the convergence analysis, we derive in the next section a bound  between the $\W_2$-distance and the $\TV$-distance between two general log-concave probability measures, which can then be extended to a $\W_2$-$\KL$ bound using Pinsker's inequality.

\subsection{Relation Between $\W_2$- and $\TV$-Distances}

When  $\mu$ and $\nu$ are both compactly supported on an Euclidean ball of diameter $D$, then it is well-known that $\W_2(\mu,\nu) \leq D \sqrt{\|\mu-\nu\|_{\TV}}$~\cite{gibbs2002choosing}. Otherwise, if $\mu$ and $\nu$ are not compactly supported, their fast-decaying tail (Lemma~\ref{lem:light_tail}) allows us to derive a similar bound, as summarized next and proved in Appendix~\ref{sec:WKLbound}.

\begin{lem}\label{lem:cor-now-lem}
\textbf{\emph{($\W_2$-$\TV$ distances inequality)}}
Let $\mu,\nu$ be log-concave probability measures on $\R^d$ both satisfying Assumption~\ref{assumption:lighttail} with $(\eta, M_\eta)$. Then, for some scalar $c \in \R$,
\begin{equation}
\begin{split}
&\W_2(\mu, \nu) \leq cd\max\l(\log\l(\frac{1}{\|\mu - \nu\|_{\TV}}\r), 1\r) \sqrt{\| \mu-\nu\|_{\TV}}.
\end{split}
\label{eq:corr-result}
\end{equation}
\end{lem}

\noindent In a sense, \eqref{eq:corr-result} is an alternative to the powerful $T_2$ inequality which does not apply generally in our setting \cite{gozlan2010transport}. Indeed, for $C_\mu>0$, recall that a probability measure $\mu$ satisfies Talagrand's $T_2(C_\mu)$ transportation inequality if 
\begin{align}
\W_2(\mu, \nu) \leq C_\mu \sqrt{\KL(\mu ; \nu)},
\label{eq:T2-recall}
\end{align}
for  any probability measure $\nu$. Above, $C_\mu$ depends only on $\mu$ and, in particular, if $\mu$ is $\kappa$ strongly log-concave,\footnote{If $\der\mu \propto e^{-f}\der x$, then we say that $\mu$ is $\kappa$ is strongly log-concave if $f$ is $\kappa$ strongly convex. } then \eqref{eq:T2-recall} holds with $C_\mu=\mathcal{O}(1/\sqrt{\kappa})$~\cite{gozlan2010transport}. In this  work, the target measures that we consider are not necessarily strongly log-concave measures, leaving us in need for a replacement to \eqref{eq:T2-recall}. 
 In our analysis, \eqref{eq:corr-result} serves as a replacement for \eqref{eq:T2-recall}. Indeed, using the Pinsker's inequality \cite{pinsker1960information}, an immediate consequence of \eqref{eq:corr-result} is that
\begin{equation}
\W_2(\mu,\nu) = \wt{\mathcal{O}} (\KL(\mu;\nu)^{\frac{1}{4}} ). 
\label{eq:W2-KL-bound}
\end{equation}

In fact, \eqref{eq:corr-result} might also be of interest in its own right, especially when working with non-strongly log-concave measures. For example, it is easy to use \eqref{eq:W2-KL-bound} to extend the well-known $\mathcal{O}(\epsilon^{-2})$ convergence rate of ULA in KL divergence to a $\wt{\mathcal{O}}(\epsilon^{-8})$ convergence rate in $\W_2$ distance in the non-strongly log-concave setting. To the best of our knowledge, such a result does not exist in the literature.

\begin{table*}[h!]
\centering
\hspace*{-5mm}
\begin{tabular}{|l||*{3}{c|}}\hline
\makebox[4em]{\small Literature} &\makebox[4em]{$\W_2$}&\makebox[4em]{{$\TV$}}&\makebox[4em]{\small $\KL$}  \\\hline \hline
\makecell{\cite{durmus2018analysis}} & - &  \small $\wt{O} \l(Ld^3\epsilon^{-4} \r)$ &  \small$\wt{O}\l(Ld^3\epsilon^{-2} \r)$  \\\hline
\makecell{\cite{durmus2017nonasymptotic}} & - &  \small $\wt{O} \l(L^2d^5\epsilon^{-2} \r)$ &  -   \\\hline
\makecell{\cite{zou2018stochastic}} & $\wt{O} \l(L^2 d^{10.5}\epsilon^{-6}\r)^*$ &  - &  -  \\\hline
\small Our work & \small $\wt{O} \l(Ld^9\epsilon^{-6}\r)$ &  \small $\wt{O} \l(Ld^3\epsilon^{-3} \r)$ &  \small$\wt{O}\l(Ld^3\epsilon^{-\frac{3}{2}} \r)$ \\\hline
\end{tabular}
\vspace{1mm}
\caption{Complexity of sampling from a smooth and log-concave probability distribution. For each metric, the entry corresponds to the total number of iterations to use in order to reach an $\epsilon$ accuracy in the specified metric. ($^*$ For comparison purpose, we extended the proof in \cite{zou2018stochastic} in the case where the distribution satisfies the weaker assumption~\ref{assumption:lighttail}. The dimension dependence is thus different from \cite{zou2018stochastic}).}
\label{table:unconstrained}
\end{table*}

\subsection{Convergence Analysis of DL-ULA}

Having covered the necessary technical tools above, we now turn our attention to the convergence rate of Algorithm~\ref{alg:DL_ULA}. The final step to take care of is to choose the sequences $\{\gamma_k\}_{k\ge 1}$ and $\{n_k\}_{k\ge 1}$ so as to obtain the best possible convergence guarantees. We summarize our result in Theorem~\ref{thm:convergence_DL_ULA}.

\begin{thm} \label{thm:convergence_DL_ULA}
\textbf{\emph{(iteration complexity of DL-ULA)}}
Let $\mu^*$ be a $L$-smooth log-concave distribution satisfying assumption~\ref{assumption:lighttail}. Suppose that $\mu_0$ also satisfies assumption~\ref{assumption:lighttail}. 
For every $k\ge 1$, let
\begin{align}
n_k = L M^2 d k^2e^{3k} 
\end{align}
\begin{align}
\g_k = \frac{1}{Ld} e^{-2k} 
\end{align}
\begin{align}
\tau_k = M k. 
\label{eq:thresh-req-thm}
\end{align}
where $M = \sqrt{\frac{2d(d+1)}{\eta^2} + M_\eta^2} = \mathcal{O}(d)$.

Let $\bar{\mu}_k, \tilde{\mu}_k$ be the average distributions associated with the iterates of outer iteration $k$ of DL-ULA using the parameters above, just before and after the projection step respectively. Then, $\forall \epsilon > 0$, we have:
\begin{itemize}
\item After $N^{\KL} = \tilde{\mathcal{O}}(Ld^3\epsilon^{-\frac{3}{2}})$ total iterations, we obtain $\KL(\bar{\mu}_k; \mu^*) \leq \epsilon$.
\item After $N^{\TV} = \tilde{\mathcal{O}}(Ld^3\epsilon^{-3})$ total iterations, we obtain $\|\tilde{\mu}_k - \mu^*\|_{\TV} \leq \epsilon$. 
\item After $N^{\W_2} =  \tilde{\mathcal{O}}(Ld^9\epsilon^{-6})$ total iterations, we obtain $\W_2(\tilde{\mu}_k, \mu^*) \leq \epsilon$.
\end{itemize}
\end{thm}

\noindent A few remarks about Theorem \ref{thm:convergence_DL_ULA} are in order.

\paragraph{Geometric sequences.} Theorem \ref{thm:convergence_DL_ULA} prescribes a geometric sequence for the choice of $\{\gamma_k\}_k$ and $\{n_k\}_k$. As outer iteration counter $k$ increases, more and more ULA (inner) iterations are performed with the constant step-size $\gamma_k$. Asymptotically, we observe that the step size decreases at a rate $n^{-\frac{2}{3}}$ where $n$ is the total number of ULA iterations. This decaying rate is faster than the standard decaying rate of $n^{-\frac{1}{2}}$ for ULA ~\cite{durmus2018analysis}.

In constrast to convex optimization where a global optimum can provably be reached with constant step-size, ULA cannot converge to the target distribution $\mu^*$ when using constant step-size, since the stationary distribution of ULA iterates ~\eqref{Langevin_iter} when using a constant step size is different from the target distribution. Asymptotically, it is thus desirable to use as small a step-size as possible.



\paragraph{Projection step.} Although the initial and target distributions are both log-concave, and thus have a sub-exponential tail, the sample distributions $\bar{\mu}_k$ are not generally log-concave, and it is not clear whether they also share the sub-exponential tail property. The projection step at the end of each outer iteration provides a way to enforce the light tail property, so that we can still apply a bound similar to~\eqref{eq:corr-result}. This procedure is made clearer in the proof of the theorem.

This procedure also provides more stability in the early outer iterations where the step-size is the largest. Moreover, since $\lim_{k \rightarrow \infty} \tau_k = \infty$, the projection step asymptotically never applies in practice.


\paragraph{Convergence rate comparison} Table~\ref{table:unconstrained} summarizes various convergence rates of Langevin dynamics based methods applied to general log-concave distributions. We observe that DL-ULA achieves improved convergence guarantees either in terms of rate or dimension dependence. Compared to \cite{durmus2017nonasymptotic}, the convergence rate in $\TV$ distance is worse in terms of accuracy $\epsilon$ but enjoys much better dimension dependence, and is also better in terms of Lipschitz constant dependence.

\section{DL-MYULA for constrained sampling}

We now apply the same multistage idea to an existing constrained sampling algorithm, and show that it allows both to obtain an asymptotic convergence and improved convergence guarantees.

\subsection{DL-MYULA algorithm}

Consider sampling from a log-concave distribution over a convex set $\Omega \subset \R^d$, i.e.,
\begin{equation}
\mu^*(x)
= 
\begin{cases}
e^{-f(x)} / \int_\Omega e^{-f(x')}dx' & x\in \Omega \\
0 & x\notin \Omega.
\end{cases}
\label{eq:constrained-dist}
\end{equation}

In \cite{durmus2018efficient,brosse2017sampling}, the authors propose to reduce this problem to an unconstrained sampling problem by penalizing the domain outside $\Omega$ directly inside the probability measure using its Moreau-Yoshida envelop. More precisely, they propose to sample from the following unconstrained probability measure $\der \mu_\lam(x) \propto e^{-f_\lam(x)} \der x$ where $f_\lam:\R^d \rightarrow \R$ is defined as:
\begin{align}
f_\lam(x) = f(x) + \frac{1}{2\lam}\|x - \text{proj}_\Omega(x)\|_2^2,
\qquad \forall x\in \R^d,
\label{eq:penalized_dist}
\end{align}

where $\text{proj}_\Omega:\R^d \rightarrow \Omega$ is the standard projection operator onto $\Omega$ defined as $\text{proj}_\Omega(x) = \argmin_{y \in \Omega} \|x-y\|_2$. Note that this penalty is easily differentiable as soon as the projection onto $\Omega$ can be computed since $\nabla f_\lam(x) = \nabla f(x) + \frac{1}{\lam} (x - \text{proj}_\Omega(x))$.

By bounding the $\TV$ distance between $\mu_\lam$ and $\mu^*$, they showed that, by sampling from $\mu_\lam$ with $\lam$ small enough, it is possible to sample from $\mu^*$ with arbitrary precision. This algorithm is called Moreau-Yoshida ULA (MYULA).

Building on this approach, we can apply our double loop algorithm, by modifying both the step size as well as the penalty parameter $\lam$ between each inner loop (Algorithm~\ref{alg:DL_MYULA}).

In addition to providing improved rate, as we will show later, our algorithm also has the advantage to use a decreasing penalty parameter $\lam$ so as to guarantee asymptotic convergence of the algorithm to the target distribution. On the other hand, MYULA uses constant penalty $\lam$, and thus saturates after a certain number of iterations. Although this looks like a trivial extension, using varying penalty parameter makes the analysis more challenging since the target distribution of the algorithm is regularly changing.

\begin{algorithm}[t!]
   \caption{DL-MYULA}
   \label{alg:DL_MYULA}
\begin{algorithmic}
\STATE \textbf{Input}: Smooth constrained probability measure $\mu^*$, step sizes $\{\gamma_k\}_{k\ge 1}$, penalty parameters $\{\lam_k\}_{k\ge 1}$, number of (inner) iterations $\{n_k\}_{k\ge 1}$, initial probability measure ${\mu}_0$ on~$\R^d$, and thresholds $\{\tau_k\}_{k \ge 1}$.
\STATE \textbf{Initialization:} Draw a sample $x_0$ from the probability measure ${\mu}_{init}$.
\FOR{$k=1, \ldots$}
\STATE $x_{k,0} \leftarrow x_{k-1}$
\FOR{$n=1,\ldots,n_k$}
\STATE $x_{k,n+1} \leftarrow x_{k,n} - \gamma_k (\nabla f(x_{k,n}) + \frac{1}{\lam_k} (x_{k,n} - \text{proj}_\Omega(x_{k,n}))) + \sqrt{2\gamma_k} g_{k,n}$, where $g_{k,n} \sim \mathcal{N}(0,I_d)$.
\ENDFOR
\STATE $x_{k} \leftarrow x_{k,i}$, where $i$ is drawn from the uniform distribution on $\{1,\cdots,n_k\}$.
\IF{$\|x_{k}\|_2 > \tau_k$}
\STATE $x_{k}\leftarrow \tau_k x_{k}/\| x_{k}\|_2$. 
\ENDIF
\ENDFOR
\end{algorithmic}
\end{algorithm}

\subsection{Convergence analysis of DL-MYULA}

We now analyze the convergence of DL-MYULA. In Algorithm~\ref{alg:DL_MYULA}, both the step-size $\gamma$ and the penalty parameter $\lam$ are decreased after each outer iteration. Therefore, at each outer iteration $k$, we aim to sample from the unconstrained penalized distribution $\der \mu_{\lam_k} \propto e^{-f_{\lam_k}(x)} \der x$ where $f_{\lam_k}$ is defined in equation~\eqref{eq:penalized_dist}. 

Similarly as for DL-ULA, we will use Lemma~\ref{descent_lemma} after each outer iteration. However, since the target distribution of outer iteration is $\mu_{\lam_k}$ instead of $\mu^*$, the inequality reads as follows:
\[
\KL(\bar{\mu}_k; \mu_{\lam_k}) \leq \frac{\W_2^2(\tilde{\mu}_{k-1}, \mu_{\lam_k})}{2 \gamma_k n_k} + Ld\gamma_k.
\]
where we recall that $\bar{\mu}_k$ is the average iterate distribution of outer iteration $k$ just before the projection step, and $\tilde{\mu}_k$ is the one just after the projection step.

In order to use a similar recursion argument as previously, we must thus bound $\W_2(\tilde{\mu}_{k-1}, \mu_{\lam_k})$ by $\W_2(\tilde{\mu}_{k-1}, \mu_{\lam_{k-1}})$. Using the triangle inequality for $\W_2$, we have 
\begin{align*}
\W_2(\tilde{\mu}_{k-1}, \mu_{\lam_k}) \leq &\W_2(\tilde{\mu}_{k-1}, \mu_{\lam_{k-1}}) + \W_2(\mu_{\lam_{k-1}}, \mu^*) \\
&+ \W_2(\mu_{\lam_k}, \mu^*).
\end{align*}

In ~\cite{brosse2017sampling}, the authors showed a bound for $\|\mu_{\lam} - \mu^*\|_{\TV}$ in terms of $\lambda > 0$, and it is easy to extend their proof to obtain a bound for $\W_2(\mu_{\lam}, \mu^*)$ (see Lemma~\ref{lem:W2-constraint-conv} and its proof in Appendix~\ref{sec:smooth-convex-set}). 

In order to prove our result, we make the same assumptions on the constraint set $\Omega$ as in ~\cite{brosse2017sampling}:

\begin{assumption}\label{assumption:constraint}
There exist $r, R, \Delta_1 > 0$ such that
\begin{enumerate}
\item $B(0,r) \subset \Omega \subset B(0,D)$ where $B(0,r_0) = \{y \in \R^d: \|x - y\|_2 \leq r_0 \} \ \forall r_0 > 0$,
\item $e^{\inf_{\Omega^c}(f) - \max_\Omega(f)} \geq \Delta_1$, where $\Omega^c = \R^d \textbackslash \Omega$.
\end{enumerate}
\end{assumption}

\begin{lem}\label{lem:W2-constraint-conv}
Let $\Omega \subset \R^d$ satisfy Assumption~\ref{assumption:constraint}. Then $\forall \lam < \frac{r^2}{8d^2}$,
\begin{equation}
\W_2^2(\mu_\lam, \mu^*) \leq C_\Omega^2 d \sqrt{\lam} 
\end{equation}
for some scalar $C_\Omega > 0$ depending on $D, r$ and $\Delta_1$.
\end{lem}

The proof of the previous Lemma is given in Appendix~\ref{sec:smooth-convex-set}. Using these results, the convergence proof is then very similar as for DL-ULA, and is summarized in Theorem~\ref{thm:convergence_DL_MYULA}, whose proof can be found in Appendix~\ref{sec:convergence-constrained}.

\begin{thm} \label{thm:convergence_DL_MYULA} \textbf{\emph{(iteration complexity of DL-MYULA)}}
Let $\Omega \subset \R^d$ be a convex set satisfying Assumption~\ref{assumption:constraint} and $\mu^*$  be a log-concave distribution given by ~\eqref{eq:constrained-dist} where $f$ has $L$-Lipschitz continuous gradient. For every $k \geq 1$, let
\begin{align}
\lam_k = \frac{1}{\frac{8d^2}{r^2} + de^{2k}} 
\end{align}
\begin{align}
n_k = Ld k^2e^{5k} 
\end{align}
\begin{align}
\g_k = \frac{1}{Ld} e^{-4k} 
\end{align}
\begin{align}
\tau_k = Dk 
\label{eq:thresh-req-thm}
\end{align}
for every $k\ge 1$.
Then, $\forall \epsilon > 0$, we have:
\begin{itemize}
\item After  $N^{\TV} = \mathcal{O}\left(d^{3.5} \epsilon^{-5} \right)$ total iterations, we obtain $\|\hat{\mu}_K - \mu^*\|_{\TV} \leq \epsilon$.
\item After $N^{\W_2} = \tilde{\mathcal{O}}\left(d^{3.5} \epsilon^{-10} \right)$ total iterations, we obtain $\W_2(\hat{\mu}_K, \mu^*) \lesssim \epsilon$. 
\end{itemize}
\end{thm}

We make a few comments about this convergence result.

\paragraph{Smoothness of $\mu_{\lam_k}$} One can notice that outer iterations in DL-MYULA are longer than in DL-ULA. In order to explain this choice, first observe that the Lipschitz constant associated with the penalized distribution $\mu_\lam$ grows as $\mathcal{O}\l(\frac{1}{\lam} \r)$ as $\lam$ goes to $0$. As $k$ increases and $\lam_k$ decreases, $\mu_{\lam_k}$ becomes less and less smooth.  Thus, for ULA to succeed in approximating $\mu_{\lam_k}$, the step size $\gamma_k$ of ULA iterations reduces accordingly, and the number of iterations increases. 

The choice for $\lam_k$ ensures that $\lam_k < \frac{r^2}{8d^2}$ as required for Lemma~\ref{lem:W2-constraint-conv} to be applicable.

\paragraph{Convergence rate comparison} Table~\ref{table:constrained} summarizes convergence rates in $\TV$ distance for various first-order constrained sampling algorithms. We can see that DL-MYULA outperforms existing approaches, both in terms of rate and dimension dependence.

\begin{table}[t]
\centering
\begin{tabular}{|l||*{3}{c|}}\hline
\makebox[5em]{\small Algorithm} & \makebox[5em]{\small TV} &  \makebox[5em]{\small \makecell{Literature}} \\\hline \hline
PLMC  & \small $d^{12} \wt{O} \l(\epsilon^{-12} \r)$ & \cite{bubeck2018sampling} \\\hline
MYULA & \small $d^{5} \wt{O} \l(\epsilon^{-6} \r)$  &\cite{brosse2017sampling}  \\\hline
DL-MYULA & \small $d^{3.5} \wt{O} \l(\epsilon^{-5} \r)$ & Our work \\\hline
\end{tabular}
\caption{Upper bounds on the number of iterations required in order to guarantee an error smaller than $\epsilon$ in $\TV$ distance for various constrained sampling algorithms.}
\label{table:constrained}
\end{table}


\section{Conclusion}

In this work, we proposed and analyzed a new step-size schedule for the well-known Unadjusted Langevin Algorithm. Our approach works by applying ULA successively with constant step-size, and by geometrically decreasing it after a certain number of iterations. Exploiting a new result on the relation between the $2$-Wasserstein distance and the $\TV$ distance of two log-concave distributions, we were able to prove new convergence guarantees for this procedure. We also applied our approach to an existing first-order constrained sampling, and showed improved convergence guarantees, both in terms of rate and dimension dependence.

\section{Acknowledgements}

This project has received funding from the European Research Council (ERC) under the European Union's Horizon 2020 research and innovation programme (grant agreement $\text{n}\circ$ 725594 - time-data).

This work was supported by the Swiss National Science Foundation (SNSF) under  grant number 407540$\_$167319.

This project  was sponsored by the Department of the Navy, Office of Naval Research (ONR)  under a grant number N62909-17-1-2111.

This work was supported by Hasler Foundation Program: Cyber Human Systems (project number 16066).

\bibliography{refs}
\bibliographystyle{icml2020}

\newpage
\appendix
\onecolumn

\section{Proof of Lemma~\ref{lem:cor-now-lem} \label{sec:WKLbound}}

Before proving Lemma~\ref{lem:cor-now-lem}, we first prove some intermediate Lemmas.

\begin{lem}
Let $\mu, \nu$ be any two distributions. Then, $\forall R > 0$, we have
\begin{align*}
\W_2^2(\mu, \nu) \leq &4R^2 \|\mu - \nu \|_{\TV} + 2 \E_{X \sim \mu}\l[ \|X\|_2^2 1_{\{\|X\|_2 > R\}}  \r] + 2R^2 \E_{X \sim \mu}\l[  1_{\{\|X\|_2 > R\}} \r] \\
& + 2 \E_{Y \sim \nu}\l[ \|Y\|_2^2 1_{\{\|Y\|_2 > R\}}  \r] + 2R^2 \E_{Y \sim \nu}\l[  1_{\{\|Y\|_2 > R\}} \r]
\end{align*}
where $1_{\{\|X\|_2 > R\}}$ is the indicator function of the set $B(0,R)^c = \{x\in \R^d : \|x\|_2 > R\}$.
\label{W-KL-bound-1}
\end{lem}

\begin{proof}
Let $X \sim \mu, Y \sim \nu$. $\W_2$-distance between probability measures $\mu$ and $\nu$ can be interpreted as the most cost-efficient transport plan to transform $\mu$ into $\nu$, defined as 
\begin{align}
\W_2^2(\mu,\nu) = \min_{(X,Y)\sim \g} \E \|X-Y\|_2^2,
\label{eq:defn_w}
\end{align}
where the minimization is over all probability measures $\g$ that marginalize to $\mu,\nu$, namely, 
\begin{align}
\g(A\times \R^d) = \mu(A),
\qquad 
\g(\R^d\times B) = \nu(B),
\end{align}
for any measurable sets $A,B\subseteq \R^d$.  For a fixed such measure $\g$, let us decompose the right-hand side of \eqref{eq:defn_w} as 
\begin{align}
\E \| X-Y\|_2^2 & = 
\E \l[\| X-Y \|_2^2 1_{E_R} \r] + \E \l[ \| X-Y \|_2^2 1_{{E}^c_R} \r],
\label{eq:decomp}
\end{align}
where $1_{E_R}$ stands for the indicator of the event $E_R=\{ \|X\|_2 \le R,\, \|Y\|_2 \le R \}$. Above, ${E_R^c}$ is the complement of $E_R$. 
For the first expectation on the right-hand side above, we write that
\begin{align}
\E \l[ \|X-Y\|_2^2 1_{E_R} \r]  
& \le 4R^2 \E \l[  1_{X\ne Y} 1_{E_R} \r] \nonumber\\
& \le 4R^2 \E [ 1_{X \ne Y}].
\label{eq:first-leg-decomp}
\end{align}
For the second expectation on the right-hand side of \eqref{eq:decomp}, we write that 
\begin{align}
\E\l[ \| X-Y\|_2^2 1_{{E}^c_R } \r] 
&  \le2 \E\l[ \| X\|_2^2  1_{{E}^c_R } \r] + 2 \E\l[ \| Y\|_2^2  1_{{E}^c_R } \r].
\qquad ( (a+b)^2 \le 2a^2 +2b^2 ) 
\label{eq:xysymm}
\end{align}
Let us in turn focus on, say, the first expectation on the right-hand side of \eqref{eq:xysymm}. Since 
$$
1_{{E}^c_R} = 1_{\{\|X\|_2>R\}} + 1_{\{\|X\|_2 \le R\}} 1_{\{\|Y\|_2 > R\}},
$$ 
we can write that 
\begin{align}
 \E \l[ \|X\|_2^2 1_{{E}^c_R } \r] 
& = \E \l[ \|X\|_2^2 1_{\{\|X\|_2 > R\}}  \r] + \E\l[ \|X\|_2^2  1_{\{\|X\|_2 \le R\}} 1_{\{\|Y\|_2 > R\}} \r] \nonumber\\
& \le \E \l[ \|X\|_2^2 1_{\{\|X\|_2 > R\}}  \r] + R^2 \E\l[  1_{\{\|Y\|_2 > R\}} \r].
\label{eq:brk-down-indic}
\end{align}

Bounding $ \E \l[ \|Y\|_2^2 1_{{E}^c_R } \r] $ similarly, we obtain
\begin{align*}
\E \| X-Y\|_2^2 \leq & 4R^2 \E [ 1_{X \ne Y}] + 2 \E_{X \sim \mu}\l[ \|X\|_2^2 1_{\{\|X\|_2 > R\}}  \r] + 2R^2 \E_{X \sim \mu}\l[  1_{\{\|X\|_2 > R\}} \r] \\
& + 2 \E_{Y \sim \nu}\l[ \|Y\|_2^2 1_{\{\|Y\|_2 > R\}}  \r] + 2R^2 \E_{Y \sim \nu}\l[  1_{\{\|Y\|_2 > R\}} \r]
\end{align*}
 
The result is then obtained by minimizing the above inequality over all coupling $\gamma$, and using the fact that $\|\mu - \nu\|_{\TV} = \min_{(X,Y)\sim \gamma} \E [ 1_{X \ne Y}]$ ~\cite{gibbs2002choosing}.
\end{proof}

\begin{lem}
\label{W-KL-bound}
Suppose that  $\mu,\nu$  both satisfy Assumption~\ref{assumption:lighttail} with $\eta, M_\eta>0$ and such that $\E_{X\sim \mu}\l[\|X\|_2^2\r], \E_{Y\sim \nu}\l[\|Y\|_2^2\r] \leq C^2$. Then, for any $R\ge C$,
\begin{align}
& \W_2^2(\mu, \nu) \leq
4R^2 \|\mu - \nu\|_{\TV} + 8 \l(  R^2 + RC+ C^2 \r) e^{-\frac{R}{C} + 1}.
\label{eq:nu-loosely-subexp}
\end{align}
\end{lem}

\begin{proof}
We start from the result of Lemma~\ref{W-KL-bound-1}. The goal is then to bound the each term on the right hand side using the tail property of log-concave distributions (Lemma~\ref{lem:moment-tail-property}).

We have
\begin{align}
\E\l[ \| X\|_2^2 1_{\{\|X\|_2 > R\}} \r]
& = 2 \int_{\|x\|_2 > R} \int_{z\in \R} 1_{\{ \|x\|_2 \ge z\}}  z  dz d\mu(x) \nonumber\\
& = 2 \int_{z\in \R} zdz \int_{\|x\|_2 \ge \max(R,z)} d\mu(x) \nonumber\\
& = 2\int_{z\in \R} z  \Pr\l[ \|X\|_2 \ge \max(R,z)\r] dz \nonumber\\
& = 2 \Pr[\|X\|_2 \ge R ] \int_{0}^R z  dz + 2\int_{R}^{\infty} z \Pr[\|X\|_2 \ge z  ] dz \nonumber\\
& \leq  R^2 e^{-\frac{R}{C}+1} + 
2 \int_{R}^{\infty} z e^{-\frac{z}{C}+1} dz  \nonumber\\
& \leq \l(  R^2 + 2C R+ 2C^2 \r) e^{-\frac{R}{C}+1}.
\label{eq:second-leg-decomp}
\end{align}

Similarly, we have
\begin{equation}
 \E[  1_{\{\|X\|_2 >R\}} ] = \Pr[ \|X|_2 > R] \leq e^{-\frac{R}{C}+1}. 
\label{eq:third-leg-decomp}
\end{equation}

Doing the same calculation for $Y$ and replacing the terms in Lemma~\ref{W-KL-bound-1} provides the result.
\end{proof}

Using the previous Lemma, it is now easy to prove the result of Lemma~\ref{lem:cor-now-lem}.

\noindent\emph{Proof of Lemma~\ref{lem:cor-now-lem}.}
Let us apply Lemma~\ref{W-KL-bound} using
$$
R = C \max\l(\log\l(\frac{1}{\| \mu - \nu\|_{\TV}} \r), 1\r).
$$ 
With this choice of $R$ and if $\|\mu-\nu\|_{\TV} \leq 1$, note that 
\begin{align}
e^{-\frac{R}{C}}  = \| \mu- \nu \|_{\TV}. 
\label{eq:min-case1}
\end{align}
On the other hand, if $\|\mu-\nu\|_{\TV} > 1$, then 
\begin{align}
e^{-\frac{R}{C}}  \le 1 \le \|\mu-\nu\|_{\TV}. 
\label{eq:min-case2}
\end{align}
Thus, Lemma~\ref{W-KL-bound} gives
\begin{align}
\W_2^2(\mu,\nu) & \leq 4C^2 \max\l(\log^2 \l( \frac{1}{\|\mu-\nu\|_{\TV}}\r), 1 \r) \|\mu-\nu\|_{\TV} + 8 C^2 \l(1 + \max\l(\log \l( \frac{1}{\|\mu-\nu\|_{\TV}}\r), 1\r)\r)^2 \|\mu-\nu\|_{\TV} \nonumber\\
& \leq  20 C^2 \max\l(\log^2 \l( \frac{1}{\|\mu-\nu\|_{\TV}}\r), 1\r) \|\mu-\nu\|_{\TV}.
\label{eq:exactW2TV}
\end{align}
Lemma~\ref{lem:cor-now-lem} then follows from taking the square root of ~\eqref{eq:exactW2TV} and using $C^2 = \frac{d(d+1)}{\eta^2} + M_\eta$ according to Lemma~\ref{lem:moment_bound}.
\\

\section{Proof of Theorem~\ref{thm:convergence_DL_ULA}}

We start by showing the following result in the case where the target distribution $\mu^*$ satisfies $\E_{X\sim \mu^*}\l[\|X\|_2^2\r] \leq 1$.

\begin{thm} \label{thm:convergence_DL_ULA_light_tail}
\textbf{\emph{(iteration complexity of DL-ULA)}}
Let $\mu^*$ be a $L$-smooth log-concave distribution such that $\E_{X\sim \mu^*}\l[\|X\|_2^2\r] \leq 1$. Suppose that $\mu_0$ also satisfies $\E_{X\sim \mu_0}\l[\|X\|_2^2\r] \leq 1$. 
For every $k\ge 1$, let
\begin{align}
n_k = L d k^2e^{3k} 
\end{align}
\begin{align}
\g_k = \frac{1}{Ld} e^{-2k} 
\end{align}
\begin{align}
\tau_k = k. 
\label{eq:thresh-req-thm}
\end{align}
Then, $\forall \epsilon > 0$, we have:
\begin{itemize}
\item After $N^{\KL} = \tilde{\mathcal{O}}(Ld\epsilon^{-\frac{3}{2}})$ total iterations, we obtain $\KL(\tilde{\mu}_k; \mu^*) \leq \epsilon$ where $\tilde{\mu}_k$ is the distribution associated to the iterates of outer iteration $k$ just \emph{before} the projection step.
\item After $N^{\TV} = \tilde{\mathcal{O}}(Ld\epsilon^{-3})$ total iterations, we obtain $\|\tilde{\mu}_k - \mu^*\|_{\TV} \leq \epsilon$. 
\item After $N^{\W_2} =  \tilde{\mathcal{O}}(Ld\epsilon^{-6})$ total iterations, we obtain $\W_2(\tilde{\mu}_k, \mu^*) \leq \epsilon$.
\end{itemize}
\end{thm}

\begin{proof}
Recall that in Algorithm~\ref{alg:DL_ULA}, we denote as $\bar{\mu}_k$ the average of the distributions associated to the iterates of outer iteration $k$ just \emph{before} the projection step, i.e., just before the projection step, $x_k \sim \bar{\mu}_k$. We also denote as $\tilde{\mu}_k$ the same distribution, but \emph{after} the projection step, i.e. the iterate that will be used as a warm start for the next outer iteration.

In order to show the result, we will show by induction that $\forall k \geq 1$,
\begin{equation}
\|\tilde{\mu}_k - \mu^*\|_{\TV} \leq u_k e^{-k}
\label{eq:recurence_TV}
\end{equation}
where $\{u_k\}_{k \geq 1}$ is a real-valued sequence defined as $u_1 = \min(2\sqrt{e} W_2(\mu_0, \mu^*) + 1+2\sqrt{2}, 2e)$ and $u_k = 4\sqrt{e u_{k-1}} + 9 + 2\sqrt{2}$.

Let us fix $k \geq 2$. Thanks to the inequality~\eqref{DL_descent_lemma},
\begin{align}
\|\bar{\mu}_k - \mu^*\|_{\TV} &\leq \sqrt{2KL(\bar{\mu}_k ; \mu^*)} \qquad \text{(Pinsker's inequality)} \nonumber \\
&\leq \sqrt{\frac{W_2^2(\tilde{\mu}_{k-1}, \mu^*)}{\gamma_k n_k} + 2L d \gamma_k} \nonumber \\
&\leq \frac{W_2(\tilde{\mu}_{k-1}, \mu^*)}{\sqrt{\gamma_k n_k}} + \sqrt{2L d \gamma_k}
\label{eq:recursion-error-light}
\end{align}
 
In order to use a recursion argument, we need to bound $W_2(\tilde{\mu}_{k-1}, \mu^*)$ by $\|\tilde{\mu}_{k-1} - \mu^*\|_{\TV}$. Note that the projection step for $\tilde{\mu}_{k-1}$ with $\tau_{k-1} = (k-1)$ ensures that $\text{Pr}_{X \sim \tilde{\mu}_{k-1}}(\|X\|_2 \geq k-1) = 0$. Knowing that $\E_{X\sim \mu^*}\l[\|X\|_2^2\r] \leq 1$, we can apply Lemma~\ref{W-KL-bound} on $W_2(\tilde{\mu}_{k-1}, \mu^*)$ using $R = k$. Also, by replacing the values for $\gamma_k, n_k$, we get

\[
W_2^2(\tilde{\mu}_{k-1}, \mu^*) \leq 4 k^2 \|\tilde{\mu}_{k-1} - \mu_{k-1}\|_{\TV} + 16  e k^2 e^{-k}.
\]

Thus,
\begin{align*}
\|\bar{\mu}_k - \mu^*\|_{\TV} &\leq \frac{2k\|\tilde{\mu}_{k-1}- \mu^*\|_{\TV} + 4\sqrt{e}k e^{-\frac{k}{2}}}{ke^{\frac{k}{2}}} + \sqrt{2}e^{-k}
\end{align*}

Now, by using the recursion hypothesis, i.e. that $\|\tilde{\mu}_{k-1} - \mu^*\|_{\TV} \leq u_{k-1} e^{-k+1}$, we have:
\begin{equation}
\|\bar{\mu}_k - \mu^*\|_{\TV} \leq \l( 2\sqrt{e u_{k-1}} + 4\sqrt{e} + \sqrt{2}\r)e^{-k}
\label{eq:recursion_ULA_1}
\end{equation}

Then, by taking into account the projection step at the end of outer iteration $k$, we obtain
\begin{align}
\|\tilde{\mu}_k - \mu_k \|_{\TV} & \le \|\tilde{\mu}_k - \bar{\mu}_k \|_{\TV} + \|\bar{\mu}_k - \mu^* \|_{\TV}
\qquad \text{(triangle inequality)} \nonumber\\
& = \Pr_{X\sim \bar{\mu}_k}[ \|X\|_2 >  \tau_k] +  \|\bar{\mu}_k - \mu^* \|_{\TV},
\label{eq:projection}
\end{align}
where the last line above follows because the projection step ensures $\Pr_{X\sim \tilde{\mu}_k}[ \|X\|_2 >  \tau_k] = 0$.  In turn, to compute the probability in the last line above, we write that 
\begin{align}
\Pr_{X\sim \bar{\mu}_k}[ \|X\|_2 \ge \tau_k ] & \le \Pr_{X\sim \mu^*}[ \|X\|_2 \ge \tau_k ] + \l| \bar{\mu}_k([\tau_k,\infty]) - \mu^* ([\tau_k,\infty]) \r|
\qquad \text{(triangle inequality)}
\nonumber\\
& \leq e^{-k} + \| \bar{\mu}_k - \mu^* \|_{\TV},
\label{eq:projection_tail}
\end{align}

By combining ~\eqref{eq:recursion_ULA_1}, ~\eqref{eq:projection} and ~\eqref{eq:projection_tail}, we finally obtain
\begin{align*}
\|\tilde{\mu}_k - \mu^*\|_{\TV} &\leq 2\|\bar{\mu}_k - \mu^*\|_{\TV} + e^{-k} \\
&\leq \l( 4\sqrt{e u_{k-1}} + 9 + 2\sqrt{2}\r)e^{-k} \\
&= u_k e^{-k}
\end{align*}

Finally, using equations ~\eqref{eq:recursion-error-light}, ~\eqref{eq:projection} and ~\eqref{eq:projection_tail} applied at $k=1$, we can also apply Lemma~\ref{W-KL-bound} and we get:
\begin{equation}
\|\tilde{\mu}_1 - \mu_1\|_{\TV} \leq \left(2 W_2(\mu_0, \mu^*) + 2\sqrt{2} + 1\right) e^{-1}
\end{equation}
which proves the result for the initial case. We thus showed that equation~\eqref{eq:recurence_TV} holds for all $k \geq 1$.

It is easy to verify that the sequence $\{u_k\}_{k\ge 1}$ converges, and is upper bounded by $U = \max(u_1, u^*)$ where $u^* = \lim_{k \rightarrow \infty} u_k$. Moreover, since $\E_{X\sim \mu^*}\l[\|X\|_2^2\r], \E_{X\sim \mu_0}\l[\|X\|_2^2\r] \leq 1$ we have that $W_2(\mu_0, \mu^*) \leq 2$, and thus $U$ is dimension independent.

After each outer iteration $k$, we thus have $\|\tilde{\mu}_k - \mu^*\|_{\TV} \leq U e^{-k}$. Therefore, after $K^{\TV} = \log(\frac{U}{\epsilon})$ iterations, we have $\|\tilde{\mu}_k - \mu^*\|_{\TV} \leq \epsilon$. The total number of iterations required is
\begin{align*}
N^{\TV} &= \sum_{k=1}^{K^{\TV}} n_k \\
&\leq LdK^2 \sum_{k=1}^{K^{\TV}} e^{3k} \\
&= \frac{1}{1-e^{-3}} Ld\log^2\l(\frac{U}{\epsilon}\r) U^3 \epsilon^{-3}
\end{align*}

Similarly, we also have $\W_2^2(\tilde{\mu}_k, \mu^*) \leq 4k^2 \|\tilde{\mu}_k - \mu^*\|_{\TV} + 16ek^2 e^{-k} \leq (4U + 16e)k^2e^{-k}$. Thus, after $K^{W_2} = \log(\frac{4U+16e}{\epsilon^2})$ iterations, we have $\W_2^2(\tilde{\mu}_k, \mu^*) \leq \epsilon \log(\frac{4U+16e}{\epsilon^2})$. The total number of iterations required is $N^{\W_2} = \mathcal{O}(Ld\epsilon^{-6})$.

Finally, we have $\KL(\bar{\mu}_k; \mu^*) \leq \frac{W_2^2(\tilde{\mu}_{k-1}, \mu^*)}{2\gamma_k n_k} + Ld\gamma_k \leq 2 \|\tilde{\mu}_{k-1} - \mu^*\|_{\TV} e^{-k} + e^{-2k} \leq (U + 1)e^{-2k}$. Therefore, after $K^{\KL} = \frac{1}{2}\log(\frac{U+1}{\epsilon})$ iterations, we have $\KL(\bar{\mu}_k; \mu^*) \leq \epsilon $. The total number of iterations required is $N^{\KL} = \mathcal{O}(Ld\epsilon^{-\frac{3}{2}})$.

\end{proof}

In order to show the more general theorem~\ref{thm:convergence_DL_ULA}, we must get rid of the assumption that $\E_{X\sim \mu^*}\l[\|X\|_2^2\r] \leq 1$. To this end, we will suppose that we apply DL-ULA to a contracted version of $\mu^*$, for which theorem~\ref{thm:convergence_DL_ULA} applies. Then, we will dilate the obtained sample in order to recover samples from the desired measure $\mu^*$ and bound the error induced by this dilatation in order to obtain the final convergence result.

Let us first recall the notion of push-forward measure. 

\begin{defn}
Let $h:\R^d \rightarrow \R$ be a strongly convex function whose gradient is denoted as $\nabla h : \R^d \rightarrow \R^d$. We say that $\nu$ is the \emph{push-forward measure} of $\mu$ under $\nabla h$, and we write $\nu = \nabla h \# \mu$, if $\nu$ is the distribution obtained by sampling from $\mu$, and then applying the map $\nabla h$ to the samples.

More precisely, it means that for every Borel set $E$ on $\R^d$, we have $\nu(E) = \mu(\nabla h^{-1}(E))$. 
\end{defn}

\begin{lem} \label{lem:push-forward}
Let $\der \mu = e^{-f(x)} \der x$ and $\der \nu = e^{-g(x)} \der x$ be such that $\nu = \nabla h \# \mu$ for some strongly convex function $h$. Then, the triplet $(\mu, \nu, h)$ must satisfy the Monge-Amp\`ere equation:
\[
e^{-f} = e^{-g \circ \nabla h} \det \nabla^2 h.
\]
\end{lem}

Let $\der \mu^* = e^{-f(x)} \der x$ be an $L$-smooth log-concave target distribution such that $\E_{X\sim \mu^*}\l[\|X\|_2^2\r] \leq M^2$. Instead of directly sample from $\mu^*$, suppose that we sample from the shrunk distribution $\nu^* = \nabla h \# \mu^*$ with $h(x) = \frac{1}{2M} \|x\|_2^2$ for some $M \geq 0$, i.e., $\nabla h(x) = \frac{x}{M}$. In this particular case, we have that $\det \nabla^2 h(x)$ is independent of $x$. Therefore, we have according the Lemma~\ref{lem:push-forward} that $\der \nu^* \propto e^{-f(M x)} \der x$. 

This means that $\nu^*$ is the same distribution as $\mu^*$, after the samples have been divided by $M$. It is easy to see that this scaling procedure implies that $\E_{X \sim \nu^*} \l[\|X\|_2 \r] = \frac{1}{M} \E_{X \sim \mu^*} \l[\|X\|_2 \r] \leq 1$. 


Thus, if we apply DL-ULA for sampling from $\nu^*$, then we can apply the convergence result provided by theorem~\ref{thm:convergence_DL_ULA_light_tail}. Note that this push-forward implies that $\nu^*$ is $M^2 L$-smooth, i.e., the Lipschitz constant has been multiplied by $M^2$. Indeed, if $g(x) = f(Mx)$ and $f$ is $L$-smooth, then,
\begin{align*}
\|\nabla g(y) - \nabla g(x)\|_2 &= M \|\nabla f(My) - \nabla f(Mx)\|_2 \\
&\leq M^2\|y - x\|_2.
\end{align*}

 Let $\tilde{\nu}$ be the approximated distribution obtained using DL-ULA on $\nu$ with $n_k = LM^2dk^2 e^{3k}$, $\gamma_k = \frac{1}{LM^2d} e^{-2k}$ and $\tau_k = k$. Then, according to Theorem~\ref{thm:convergence_DL_ULA_light_tail}, we have the following convergence results:
\begin{itemize}
\item After $N^{\KL} = \tilde{\mathcal{O}}(LM^2d \epsilon^{-\frac{3}{2}})$ total iterations, we obtain $\KL(\tilde{\nu} - \nu^*) \leq \epsilon$.
\item After $N^{\TV} = \tilde{\mathcal{O}}(LM^2d \epsilon^{-3})$ total iterations, we obtain $\|\tilde{\nu} - \nu^*\|_{\TV} \leq \epsilon$.
\item After $N^{\W_2} = \tilde{\mathcal{O}}(LM^2d \epsilon^{-6})$ total iterations, we obtain $\W_2(\tilde{\nu}, \nu^*) \leq \epsilon$.
\end{itemize}

By applying the inverse mapping $\nabla h^{-1} (x) = M x$, we obtain samples from $\tilde{\mu} = \nabla h^{-1} \# \tilde{\nu}$. Interestingly, it can be shown that applying the same push-forward on two measures does not change their $\TV$-distance not their $\KL$ divergence ~\cite{hsieh2018mirrored}:
\begin{align*}
&\|\tilde{\nu} - \nu^*\|_{\TV} = \|\nabla h^{-1} \#\tilde{\nu} - \nabla h^{-1} \#\nu^*\|_{\TV} = \|\tilde{\mu} - \mu^*\|_{\TV}, \\
&\KL(\tilde{\nu}; \nu^*) = \KL(\nabla h^{-1} \#\tilde{\nu}; \nabla h^{-1} \#\nu^*) = \KL(\tilde{\mu}; \mu^*).
\end{align*}

In terms of $\W_2$-distance, when applying the same mapping $\nabla h^{-1}$ to two measures, it can be shown that
\[
\W_2(\tilde{\mu}; \mu^*) \leq M \W_2(\nabla h \#\tilde{\mu}; \nabla h \#\mu^*) = M \W_2(\tilde{\nu}; \nu^*).
\]

Therefore, by sampling from $\nu^*$, and then multiplying the obtained samples by $M$, we obtain the following convergence results:
\begin{itemize}
\item After $N^{\KL} = \tilde{\mathcal{O}}(LM^2d \epsilon^{-\frac{3}{2}})$ total iterations, we obtain $\KL(\tilde{\mu} - \mu^*) \leq \epsilon$.
\item After $N^{\TV} = \tilde{\mathcal{O}}(LM^2d \epsilon^{-3})$ total iterations, we obtain $\|\tilde{\mu} - \mu^*\|_{\TV} \leq \epsilon$.
\item After $N^{\W_2} = \tilde{\mathcal{O}}(LM^2d \l(\frac{\epsilon}{M}\r)^{-6}) = \tilde{\mathcal{O}}(LM^8 d \epsilon^{-6})$ total iterations, we obtain $\W_2(\tilde{\mu}, \mu^*) \leq \epsilon$.
\end{itemize}

Finally, we make the following important observation. By modifying the parameters $\gamma_k, \tau_k$, it is possible to mimic the above procedure by directly applying DL-ULA to $\mu^*$. Suppose that we apply DL-ULA for sampling from $\der \nu^* = e^{g(y)} \der y$, where $g(y) = f(My)$, using parameters $\gamma_k, n_k, \tau_k$. Let $y_i$ be the iterates of some arbitrary outer iteration $k$, and let $x_i = M y_i$ be their scaled version. The ULA iterates are:
$$
\left\{
\begin{array}{ll}
        y_{i+1} = y_i + \gamma_i \nabla g(y_i) + \sqrt{2 \tilde{\gamma_i}} g_i \\
       x_{i+1} = M y_{i+1}
    \end{array}
\right.
$$
Since $\nabla g(y_i) = M \nabla f(M y_i)$, we can rewrite this scheme only in terms of $\{x_i\}$:
\[
x_{i+1} = x_i + M^2 \gamma_i \nabla f(x_i) + \sqrt{2 M^2 \gamma_i} g_i
\]

Moreover, applying the projection step to $y_i$ with parameter $\tau_k$ is the same as applying this projection to $x_i$ with parameter $M \tau_k$.

Therefore, applying DL-ULA to $\nu^*$ using parameters $n_k, \gamma_k, \tau_k$, and then multiplying the iterates by $M$ is the same as directly applying DL-ULA to $\mu^*$ using parameters $n_k, M^2 \gamma_k, M\tau_k$.

Overall, if we apply DL-ULA to a distribution $\mu^*$ such that $\E_{X\sim \mu^*}\l[\|X\|_2^2\r] \leq M^2$ using $n_k = LM^2d k^2 e^{3k}$, $\gamma_k = \frac{1}{Ld} e^{-k}$ and $\tau_k = Mk$, then we can guarantee convergence rates of $\tilde{\mathcal{O}}(LM^2d \epsilon^{-\frac{3}{2}})$, $\tilde{\mathcal{O}}(LM^2d \epsilon^{-3})$ and $\tilde{\mathcal{O}}(LM^8d \epsilon^{-6})$ in $\KL$ divergence, $\TV$-distance and $\W_2$-distance respectively.

Finally, thanks to Lemma~\ref{lem:moment_bound}, we know that we can choose $M = \sqrt{\frac{2d(d+1)}{\eta^2} + M_\eta^2} = \mathcal{O}(d)$. Thus, plugging this value inside the convergence results above concludes the theorem.

\section{Proof of Lemma~\ref{lem:W2-constraint-conv}}\label{sec:smooth-convex-set}

\begin{proof}
A similar result has been shown in \cite{brosse2017sampling} (Proposition 5) for $W_1$ distance, and it is only a matter of trivial technicalities to extend their result to $W_2$ distance. Since the full proof requires to introduce several concepts that are out of the scope of this paper, we only present the required modifications that allow us to extend the result from $\W_1$- to $\W_2$-distance.

Using \cite{villani2009optimal}, Theorem 6.15, we have:
\begin{equation}
W_2^2(\mu_\lam, \mu^*) \leq 2 \int_{\R^d} \|x\|_2^2 |\mu^*(x) - \mu_\lam(x)| dx = A + B
\end{equation}
where
\begin{equation}
A = \int_{K^c} \|x\|_2^2 \mu_\lam(x) dx \text{ ,     } B = \left(1 - \frac{\int_K e^{-f}}{\int_{\R^d} e^{-f_\lam}}\right) \int_K \|x\|_2^2 \mu^*(x) dx
\end{equation}

Following very closely the proof in \cite{brosse2017sampling} (equations 48 to 51), we can easily obtain:
\begin{equation}
A \leq \Delta_1^{-1} \sum_{i=0}^{d-1} \left(\frac{d}{r}\sqrt{\frac{\pi\lam}{2}}\right)^{d-i}\left(R^2 + 2R\sqrt{\lam(d-i+2)} +  \lam(d-i+2)\right).
\end{equation}
Therefore, for $\lam \leq \frac{r^2}{2\pi d^2}$,
\begin{equation}
A \leq \Delta_1^{-1}\sqrt{2\pi\lam}dr^{-1}\left(R^2 + 2Rr\sqrt{\frac{3}{2d\pi}} +  r^2\frac{3}{2d\pi}\right).
\end{equation}

Moreover, it is also shown in \cite{brosse2017sampling} (equations 17, 30, 42) that $\left(1 - \frac{\int_K e^{-f}}{\int_{\R^d} e^{-f_\lam}}\right) \leq \Delta_1^{-1}2\pi\lam dr^{-1}$, which implies:
\begin{equation}
B \leq \Delta_1^{-1}\sqrt{2\pi\lam}dr^{-1}R^2
\end{equation}

We thus showed that $W_2(\mu_\lam, \mu^*) \leq C \sqrt{d} \lam^{\frac{1}{4}}$ for some $C > 0$ depending on $D,r,\Delta_1$.

\end{proof}

\section{Convergence rate of HULA for sampling from a distribution over a bounded domain} \label{sec:convergence-constrained}

The proof of Theorem~\ref{thm:convergence_DL_MYULA} is very similar to the one for DL-ULA. Before presenting it, we will need an auxiliary Lemma, showing the light tail property of the distributions $\mu_\lam$. 

\begin{lem} \label{lem:lighttail_constrained}
For $\lam \leq \frac{r^2}{8 d^2}$, the distribution $\mu_\lam$ as defined in equation~\eqref{eq:penalized_dist} satisfies 
\[
\text{Pr}_{X \sim \mu_\lam} (\|X\|_2 \geq R) \leq \sigma e^{-\frac{R}{D}}
\]
for some scalar $\sigma > 0$ and any $R > 0$, where $D$ is the diameter of the constraint set $\Omega$.
\end{lem}

\begin{proof}

Suppose first that $R \geq 2D$. Then,

\begin{align*}
\Pr_{X\sim \mu_\lam}\l[ \|X\|_2 \geq R\r] &= \frac{\int_{\text{B}(0,R)^c} e^{-f(x) - \frac{1}{2\lam} \|x - \text{proj}_\Omega(x)\|_2^2} \der x}{\int_{\Omega} e^{-f(x)} \der x + \int_{\Omega^c} e^{-f(x) - \frac{1}{2\lam} \|x - \text{proj}_\Omega(x)\|_2^2} \der x} \\ 
&\leq \Delta_1 \frac{\int_{\text{B}(0,R)^c} e^{-\frac{1}{2\lam} (\|x\|_2 - D)^2} \der x}{\text{Vol}(\Omega)} \\
&\leq \Delta_1 \text{Vol}(\Omega)^{-1} \int_R^\infty u^{d-1} e^{-\frac{1}{2\lam} (u - D)^2} \der u \\
&= \Delta_1 \text{Vol}(\Omega)^{-1} d \text{Vol}(B(0,1))  \int_R^\infty u^{d-1} e^{-\frac{1}{2\lam} (u - D)^2} \der u \\
&\leq \Delta_1 d \frac{\text{Vol}(B(0,1))}{\text{Vol}(B(0,r))} D^{d-1}  \int_{R-D}^\infty (u + D)^{d-1} e^{-\frac{1}{2\lam} u^2} \der u \\
&\leq \Delta_1 d \frac{1}{r^d}  \int_{R-D}^\infty (2u)^{d-1} e^{-\frac{1}{2\lam} u^2} \der u \qquad \text{ since $u \geq R - D \geq D$} \\
&\leq \Delta_1 d \frac{1}{r^d}2^{d-1}  \int_{\frac{1}{2\lam}(R-D)^2}^\infty \l(2v\lam\r)^{\frac{d-1}{2}} e^{-v} \sqrt{\frac{\lam}{2v}}\der u  \qquad (v = \frac{1}{2\lam} u^2) \\
&\leq \Delta_1 d \frac{2^{\frac{3}{2}d-3} \lam^{\frac{d}{2}}}{r^d} \Gamma\l(\frac{d}{2}; \frac{1}{2\lam}(R-D)^2\r) \qquad \text{ where $\Gamma(s;x)$ is the incomplete Gamma function} \\
&\leq \Delta_1 d \frac{2^{-3} }{d^d} \frac{d}{2} \l(\frac{1}{2\lam}(R-D)^2\r)^{\frac{d}{2}} e^{-\frac{1}{2\lam}(R-D)^2} \qquad \text{ since for $x \geq s$, $\Gamma(s;x) \leq sx^s e^{-x}$, $\lam \leq \frac{r^2}{8 d^2}$} \\
&\leq \l( \Delta_1^{\frac{1}{d^2}} 2^{\frac{-4}{d^2}} d^{\frac{2}{d^2}} \l(\frac{(R-D)^2}{2\lam d^2}\r)^{\frac{1}{2d}} e^{-\frac{1}{2\lam d^2}(R-D)^2} \r)^{d^2} \\
&\leq \l( c_d e^{-\frac{1}{\sqrt{2\lam}d}(R-D)} \r)^{d^2} \qquad \text{ since $x e^{-x^2} \leq e^{-x}$ $\forall x \geq 0$ and $\frac{1}{2\lam d^2}(R-D)^2 \geq 1$}
\end{align*}

where in the last line, $c_d = \Delta_1^{\frac{1}{d^2}} 2^{\frac{-4}{d^2}} d^{\frac{2}{d^2}}$. If $c_d e^{-\frac{\sqrt{\frac{1}{2\lam}}}{d}(R-D)} \geq 1$, then, this does not provide a useful bound, and we can always write $\Pr_{X\sim \mu_\lam}\l[ \|X\|_2 \geq R\r] \leq 1 \leq c_d e^{-\frac{\sqrt{\frac{1}{2\lam}}}{d}(R-D)}$. On the other hand, if $c_d e^{-\frac{\sqrt{\frac{1}{2\lam}}}{d}(R-D)} \leq 1$, then we have $\Pr_{X\sim \mu_\lam}\l[ \|X\|_2 \geq R\r] \leq \l( c_d e^{-\frac{\sqrt{\frac{1}{2\lam}}}{d}(R-D)} \r)^{d^2} \leq c_d e^{-\frac{\sqrt{\frac{1}{2\lam}}}{d}(R-D)}$.

Therefore, we can write:
\begin{align*}
\Pr_{X\sim \mu_\lam}\l[ \|X\|_2 \geq R\r]  &\leq c_d e^{-\frac{\sqrt{\frac{1}{2\lam}}}{d}(R-D)} \\
&\leq c_d e^{-2(\frac{R}{D}-1)} \qquad \text{ since $\lam \leq \frac{r^2}{8 d^2} \leq \frac{D^2}{8 d^2}$} \\
&\leq \max(1,c_d) e^2 e^{-\frac{R}{D}}.
\end{align*}

Moreover, in the case $R \leq 2D$, we have $\max(1,c_d) e^2 e^{-\frac{R}{D}} \geq 1 \geq \Pr_{X\sim \mu_\lam}\l[ \|X\|_2 \geq R\r]$. We thus showed the result with $\s = \max(1,c_d) e^2$. Note that although $c_d$ depends on $d$, it is bounded and converges to $1$ as $d \rightarrow \infty$, thus it does not involve any asymptotic dependence in $d$.

\end{proof}

Using this Lemma, we can now prove our convergence result for DL-MYULA (Theorem~\ref{thm:convergence_DL_MYULA}).

\begin{proof}

Let denote $\mu_k \equiv \mu_{\lam_k}$ the target distributions of the ULA iterations at outer iteration $k \geq 1$, and $\mu_{init}$ the initial distribution. It is straightforward to show that the distributions $\mu_k$ are $L_k$-smooth with $L_k = L + \frac{1}{\lam_k}$. 

The proof goes exactly the same way as for Theorem~\ref{thm:convergence_DL_ULA}. We will show by induction that $\forall k \geq 1$,
\[
\| \tilde{\mu}_k - \mu_k\|_{\TV} \leq u_k e^{-k} + \sqrt{2 + \frac{16 d^2}{Lr^2}} e^{-2k}
\]
where $\{u_k\}_{k \geq 1}$ is defined $u_1 = \sqrt{e} \left(\W_2(\mu_{init}, \mu^* + C_\Omega d^{\frac{1}{4}}) \right)$ and the recurrence relation
\[
u_k = 4D\sqrt{e u_{k-1}} + 4D\sqrt{\sigma} + \frac{2C_\Omega d^{\frac{1}{4}}(\sqrt{e}+1)}{k^2} + \frac{2\sqrt{2} d^{\frac{1}{2}}}{L} + \sigma.
\]

For any $k \geq 1$, we have:
\begin{align}
\|\bar{\mu}_k - \mu_k\|_{\TV} &\leq \sqrt{2\KL(\bar{\mu}_k ; \mu_k)} \qquad \text{(Pinsker's inequality)} \nonumber \\
&\leq \sqrt{\frac{W_2^2(\tilde{\mu}_{k-1}, \mu_k)}{\gamma_k n_k} + 2 L_k d \gamma_k} \nonumber \\
&\leq \frac{W_2(\tilde{\mu}_{k-1}, \mu_k)}{\sqrt{\gamma_k n_k}} + \sqrt{2L_k d \gamma_k} \nonumber \\
&\leq \frac{W_2(\tilde{\mu}_{k-1}, \mu_{k-1})}{\sqrt{\gamma_k n_k}} + \frac{W_2(\mu_{k-1}, \mu^*)}{\sqrt{\gamma_k n_k}} + \frac{W_2(\mu_k, \mu^*)}{\sqrt{\gamma_k n_k}} + \sqrt{2L_k d \gamma_k}
\label{eq:recursion-error-constr}
\end{align}

For the second and third term, we can use Lemma~\ref{lem:W2-constraint-conv} and the values of $\lam_k$ to show that $\forall k \geq 1$,
\begin{equation}
W_2(\mu_k, \mu^*) \leq C_\Omega d^{\frac{1}{4}} e^{-\frac{k}{2}}
\label{eq:smoothing-constr}
\end{equation}

For the first term, we use Lemma~\ref{W-KL-bound-1} with $R = Dk$ together with the fact that $\text{Pr}_{X \sim \tilde{\mu}_{k-1}}(\|X\|_2 \geq Dk) = 0$ thanks to the projection step, and the light tail property of $\mu_k$ to obtain
\begin{equation}
W_2^2(\tilde{\mu}_{k-1}, \mu_{k-1}) \leq 4D^2 k^2 \|\tilde{\mu}_{k-1} - \mu_{k-1}\|_{\TV} + 4 D^2 k^2  \sigma e^{-k + 1}.
\label{eq:W-TV-constr}
\end{equation}

By replacing ~\eqref{eq:smoothing-constr} and ~\eqref{eq:W-TV-constr} in ~\eqref{eq:recursion-error-constr}, and using the recursion hypothesis for $\|\tilde{\mu}_{k-1} - \mu_{k-1}\|_{\TV}$, we obtain
\begin{equation}
\|\bar{\mu}_k - \mu_k\|_{\TV} \leq \left(2D \sqrt{eu_{k-1}} + 2D\sqrt{\s} + \frac{C_\Omega d^{\frac{1}{4}}(\sqrt{e} + 1)}{k^2} + \frac{\sqrt{2}d^{\frac{1}{2}}}{L} \right) e^{-k} + \sqrt{2+ \frac{16d^2}{Lr^2}}e^{-2k}
\end{equation}

Similarly as for DL-ULA, and using Lemma~\ref{lem:lighttail_constrained} we can show that
\[
\|\tilde{\mu}_k - \mu_k\|_{\TV} \leq 2 \|\bar{\mu}_k - \mu_k \|_{\TV} + \sigma e^{-k}
\]

Thus, using the recurrence relation for $u_k$, we have
\begin{equation}
\|\tilde{\mu}_k - \mu_k \|_{\TV} \leq u_k e^{-k} +  \sqrt{2+ \frac{16d^2}{Lr^2}}e^{-2k}
\end{equation}

as required to show the induction property. The case for $k=1$ is shown analogous to DL-ULA.

Finally, in order to relate $\tilde{\mu}_k$ to the target distribution $\mu^*$, we use the result shown in ~\cite{bubeck2018sampling} that $\|\mu_{\lam} - \mu^*\|_{\TV} \leq C' d \sqrt{\lam}$ for some constant $C' > 0$ and $\forall \lam < \frac{r^2}{8d^2}$.

We can easily show that the sequence $\{u_k\}_{k \geq 1}$ increasingly converges to the following limit:
\begin{align*}
U &= 8eD^2 + 4D\sqrt{\sigma} + \frac{2C_\Omega d^{\frac{1}{4}}(\sqrt{e}+1)}{k^2} + \frac{2\sqrt{2} d^{\frac{1}{2}}}{L} + \sigma + 4D\sqrt{4eD^2 + 2D\sqrt{\sigma} + \frac{C_\Omega d^{\frac{1}{4}}(\sqrt{e}+1)}{k^2} + \frac{\sqrt{2} d^{\frac{1}{2}}}{L} + \frac{\sigma}{2}} \\
&= \mathcal{O}(\sqrt{d}).
\end{align*}

We thus have for all $k \geq 1$:
\[
\|\tilde{\mu}_k - \mu^* \|_{\TV} \leq (U+C' \sqrt{d}) e^{-k} + \sqrt{2+ \frac{16d^2}{Lr^2}}e^{-2k}
\]

Therefore, after $K^{\TV} = \log \l(\frac{2\max\l(U+C' \sqrt{d}, \l(2+ \frac{16d^2}{Lr^2}\r)^{\frac{1}{4}}\r)}{\epsilon} \r)$ iterations, we have $\|\tilde{\mu}_k - \mu^*\|_{\TV} \leq \epsilon$. The total number of iterations required is $N^{\TV} = \tilde{\mathcal{O}}(Ld^{3.5}\epsilon^{-5})$.

Finally, using $\W_2^2(\tilde{\mu}_k, \mu^*) \leq 4D^2k^2 \|\tilde{\mu}_k - \mu^*\|_{\TV}$, we can obtain a similar convergence result, i.e., after $K^{\W_2} = \log \l(\frac{8D^2\max\l(U+C' \sqrt{d}, \l(2+ \frac{16d^2}{Lr^2}\r)^{\frac{1}{4}}\r)}{\epsilon} \r)$ iterations, we have $\W_2(\tilde{\mu}_k, \mu^*) \leq \epsilon \log^2(K)$. The total number of iterations required is $N^{\W_2} = \tilde{\mathcal{O}}(Ld^{3.5}\epsilon^{-10})$.

\end{proof}

\end{document}